\documentclass{article}
\usepackage{eurosym}
\usepackage{hyperref}
\usepackage{amssymb}
\usepackage{amsfonts}
\usepackage{graphicx}
\usepackage{amsmath}
\usepackage{float,color,ulem}
\usepackage{epsf,epsfig,subfigure}
\usepackage{float,color,ulem}
\usepackage{ulem}
\usepackage{bbm}

\newfloat{figure}{H}{lof}
\newfloat{table}{H}{lot}
\floatname{figure}{\figurename}
\floatname{table}{\tablename}
\newtheorem{theorem}{Theorem}[section]

\newtheorem{assumption}[theorem]{Assumption}

\newtheorem{definition}[theorem]{Definition}

\newtheorem{lemma}[theorem]{Lemma}

\newtheorem{remark}[theorem]{Remark}

\hypersetup{colorlinks=true, linkcolor=blue, citecolor=red}
 \newenvironment{proof}[1][Proof]{\textbf{#1.} }{\ \rule{0.5em}{0.5em}}

\numberwithin{equation}{section}
\allowdisplaybreaks
\newcommand{\norm}[1]{\left\lVert#1\right\rVert}
\def\R{\mathbb{R}}
\def\C{\mathbb{C}}
\def\N{\mathbb{N}}
\def\Re{{\rm Re}\,}

\begin{document}
		\title{\textbf{A short proof for Hopf bifurcation in Gurtin-MacCamy's population dynamics model}}
	\author{\textsc{Arnaud Ducrot$^{(a)}$ Hao Kang$^{(a)}$ and Pierre Magal$^{(b)}$\thanks{Corresponding author. e-mail: \href{mailto:pierre.magal@u-bordeaux.fr}{pierre.magal@u-bordeaux.fr}}  }\\
		{\small \textit{$^{(a)}$Normandie Univ, UNIHAVRE, LMAH, FR-CNRS-3335, ISCN, 76600 Le Havre, France}},\\  
		{\small \textit{$^{(b)}$Univ. Bordeaux, IMB, UMR 5251, F-33400 Talence, France.}} \\
		{\small \textit{CNRS, IMB, UMR 5251, F-33400 Talence, France.}}
	}
	\maketitle
	
%	% \title[short text for running head]{full title}
%	\title[Hopf bifurcation in Gurtin-MacCamy's age-structured model]{\textbf{A short proof for Hopf bifurcation in Gurtin-MacCamy's population dynamics model}}
%	
%	
%	
%	%    author one information
%	% \author[short version for running head]{name for top of paper}
%	\author[A. Ducrot]{Arnaud Ducrot}
%	\address{Normandie Univ, UNIHAVRE, LMAH, FR-CNRS-3335, ISCN, 76600 Le Havre, France}
%	%\curraddr{}
%	\email{arnaud.ducrot@univ-lehavre.fr}
%%	\thanks{??}
%
%	\author[H. Kang]{Hao Kang}
%    \address{Normandie Univ, UNIHAVRE, LMAH, FR-CNRS-3335, ISCN, 76600 Le Havre, France}
%    %\curraddr{}
%    \email{hao.kang@univ-lehavre.fr}
%	
%	%    author two information
%	\author[P. Magal]{Pierre Magal}
%	\address{Univ. Bordeaux, IMB, UMR 5251, F-33400 Talence, France, CNRS, IMB, UMR 5251, F-33400 Talence, France.}
%	\email{pierre.magal@u-bordeaux.fr}
%	%\thanks{}
%	
%	%    \subjclass is required.
%	%\subjclass[2020]{Primary 34K20; 37C10; 37N25; 92D25 }
%	% \subjclass[2020]{{34K20; 37C10; 37N25; 92D25 } }
%	
%	\date{}
%	
%	\dedicatory{}
	
	%    "Communicated by" -- provide editor's name; required.
%	\commby{Wenxian Shen}
	
%	\keywords{Age structure; Population dynamics; Hopf bifurcation.}
	
	%    Abstract is required.
	\begin{abstract}
	In this paper, we provide a short proof for the Hopf bifurcation theorem in the Gurtin-MacCamy's population dynamics model. Here we use the  Crandall and Rabinowitz's approach, based on the implicit function theorem. Compared with previous methods, here we require the age-specific birth rate to be slightly smoother (roughly of bounded variation), but we have a huge gain for the length of the proof.
%
%    	In this paper, we provide a short proof for the Hopf bifurcation theorem in Gurtin-MacCamy's population dynamics model. Here we use the  Crandall and Rabinowitz's approach. Compared with previous methods, it is required a little more  regularity on the birth function, but we have a huge gain for the length of the proof.  
	\end{abstract}
	
	\maketitle
	
	
	%\bibliographystyle{amsplain}

	%\noindent \textbf{Keywords:} Age structure; population dynamics; Lyapunov function; Global stability; Uniform persistence; Global attractors.}

\section{Age-Structured Models}
We consider the existence of periodic solutions for the following equation,
\begin{equation}\label{1.1}
\begin{cases}
(\partial_t+\partial_a)u(t, a)=-m \, u(t, a), \quad a\in(0, +\infty),\\
u(t, 0)=f\left(\nu, \int_0^\infty b(a)u(t, a)da\right),\\
u(0, \cdot)=\psi\in L^1_+((0, \infty), \mathbb{R}),
\end{cases}
\end{equation}
where $m >0$ and $b \in L^\infty_+(0, \infty)$ are mortality rate and fertility rate of the population respectively, the nonlinear function $f \in C^2 \left(\mathbb{R}\times\mathbb{R}, \mathbb{R} \right) $ describes the birth limitations for the population while $\nu\in \mathbb{R}$ is regarded as a bifurcation parameter. Here $u(t, a)$ denotes the density of a population at time $t$ with age $a$. This equation is referred as the \textit{Gurtin-MacCamy's age-structured equation} and was introduced in its nonlinear form by Gurtin and MacCamy in \cite{gurtin1974non} to study temporal evolution of biological populations.

%Such equations as \eqref{1.1} have been studied extensively by many researchers (Webb \cite{webb1984theory}, Iannelli \cite{iannelli1995mathematical} and Cushing \cite{cushing1998an}, Inaba \cite{inaba2017age}). 

The existence of nontrivial periodic solution induced by Hopf bifurcation has been observed in various specific age-structured models (Cushing \cite{cushing1983bifurcation, cushing1998an}, Pr{\"u}ss \cite{pruss1983qualitative}, Swart \cite{swart1988hopf}, Kostava and Li \cite{kostova1996oscillations}, Bertoni \cite{bertoni1998periodic}, Magal and Ruan \cite{magal2009center}). In this paper we shall use the implicit function theorem to establish the Hopf bifurcation theorem that is used to obtain the existence of nontrivial periodic solutions of the age-structured model \eqref{1.1}, that is, a nontrivial periodic solution bifurcated from the equilibrium of \eqref{1.1} when the bifurcation parameter $\nu$ takes some critical values. 

For two dimensional ordinary differential equations Crandall and Rabinowitz \cite{crandall1977hopf} requires less regularity ($C^2$-right hand side) than the standard result by Hale and Kocack \cite{Hale-Kocak} (which requires $C^3$-right hand side)  and Hassard, Kazarinoff and Wan \cite{Hassard-Kazarinoff-Wan} (which requires $C^4$-right hand side). Here we assume that the function $f$ is only $C^2$ which corresponds to the regularity imposed by Crandall and Rabinowitz \cite{crandall1977hopf} for their result applied to ordinary differential equations. Such a regularity assumption has been mentioned already in Liu, Magal and Ruan \cite[see Remark 2.5]{LMR11}. In the context of partial differential equations Crandall and Rabinowitz \cite{crandall1977hopf} original theorem only applies to parabolic PDE since their proof strongly uses the fact that the semigroup is generated by a sectorial operator $A:D(A) \subset X \to X $ on a Banach space $X$ and verifies that the map  
$t \to e^{A t} x$ (with $x \in X$) 
is of class $C^1$ on $(0, \infty)$. Here we are working with hyperbolic operator, therefore such a property is not satisfied. Nevertheless by exploiting the special structure of the system, and imposing some extra regularity for the birth function $a \to b(a)$ (i.e. $b$ has bounded variation on $[0, \infty)$) we are still able to apply Crandall and Rabinowitz's ideas.  

Note that the previous result by Magal and Ruan \cite{magal2009center} and Liu, Magal and Ruan \cite{LMR11} only assume that the birth function $b$ belongs to $L^\infty(0,\infty)$. Here the fact that $a \to b(a)$ has bounded variation on each $[0, \infty)$, means that  $a \to b(a)$ has a finite number of discontinuity point on each bounded interval on $[0, \infty)$ and that $b(a)$ is continuous in between two successive discontinuity points. Such an assumption is sufficient for most practical  examples. Finally this paper is closely related to the work of Cushing  \cite{cushing1983bifurcation} in which he considered an equation with age and delay at birth. In \cite{cushing1983bifurcation},  the function $a \to b(a)$ is assumed to be of class $C^1$ which is stronger than our bounded variation assumption. 

This paper is organized as follows. In Section \ref{section2}, we give the well-posedness result of \eqref{1.1}. In Section \ref{section3}, we provide the assumptions for our Hopf bifurcation theorem, while Section \ref{section4} is devoted to state and prove the Hopf bifurcation theorem.

\section{Well-posedness}\label{section2}

Set 
\begin{equation*}
X=\mathbb{R}\times L^1((0, \infty), \mathbb{R}) \text{ and } X_0=\left\{ 0_\R  \right\}\times L^1((0, \infty), \mathbb{R}).
\end{equation*}
Assume that $X$ is endowed with the product norm 
$$ 
\norm{x}=|\alpha|+\norm{\psi}_{L^1((0, \infty), \mathbb{R})}, \quad \forall x=\begin{pmatrix}
	\alpha\\\psi
\end{pmatrix}\in X. 
$$
Consider the linear operator $A: D(A)\subset X\to X$ given by
\begin{equation*}
	A\begin{pmatrix}
		0\\\psi
	\end{pmatrix}=\begin{pmatrix}
		-\psi(0)\\-\psi'- m \psi
	\end{pmatrix}
\end{equation*} 
with \begin{equation*}
	D(A)=\{0_\R\}\times W^{1, 1}((0, \infty), \mathbb{R}).
\end{equation*} 
Recall that $X_0$ is the closure of $D(A)$ in $X$. In addition, note that $A_0$, the part of $A$ in $X_0$, generates a $C_0-$semigroup of bounded linear operators, denoted by $\{T_{A_0}(t)\}_{t\ge0}$ and explicitly given by 
$$ 
T_{A_0}(t)\begin{pmatrix}
	0\\\psi
\end{pmatrix}=\begin{pmatrix}
	0\\ \widehat{T}_{A_0}(t)\psi
\end{pmatrix}, 
$$
wherein we have set, for all $t\geq 0$ and $\psi\in L^1((0,\infty),\mathbb{R})$,
$$ 
\widehat{T}_{A_0}(t)(\psi)(a)=\begin{cases}
	e^{-m\, t}\psi(a-t), \quad&\text{if}\quad a\ge t,\\
	0, \quad&\text{if}\quad a\le t. 
\end{cases} 
$$
Moreover $A$ generates an integrated semigroup of $X$, denoted by $\{S_A(t)\}_{t\ge0}$ and defined, for $t\geq 0$ by
$$ 
S_A(t)\begin{pmatrix}
	\alpha\\\psi
\end{pmatrix}=\begin{pmatrix}
	0\\L(t)\alpha+\int_0^t\widehat{T}_{A_0}(s)\psi ds
\end{pmatrix}, \;\forall \begin{pmatrix}
	\alpha\\\psi
\end{pmatrix}\in X,
$$
where  
$$ 
L(t)(\alpha)(a)=\begin{cases}
	0,\quad &\text{if}\quad a\ge t,\\
	e^{-m\, a}\alpha, \quad&\text{if}\quad a\le t.
\end{cases} 
$$
Define the map $H:\R \times  X_0\to X$ by
\begin{equation*}
	H \begin{pmatrix}
		\nu,  \begin{pmatrix}
		0\\\psi
	\end{pmatrix}
\end{pmatrix}
=\begin{pmatrix}
		f\left(\nu, \int_0^\infty b(a)\psi(a)da\right)\\0
	\end{pmatrix}.
\end{equation*} 
Then by identifying $u(t)$ with $v(t)=\begin{pmatrix}
	0\\u(t)
\end{pmatrix}$, problem \eqref{1.1} can be rewritten as the following abstract Cauchy problem
\begin{equation}\label{2.1}
	\frac{dv(t)}{dt}=Av(t)+H(\nu, v(t)),\quad t\ge0,\quad v(0)=y\in X_0. 
\end{equation}
Since $f$ is Lipschitz continuous on bounded sets, the general results proved in Magal and Ruan \cite[see section 5]{MR2007} or \cite[Chapter 5]{magal2018theory} ensure that the Cauchy problem \eqref{2.1} generated a maximal semiflow (with eventually some blowup),
denoted below by $\{U_\nu(t)\}_{t\ge0}$. 

\section{Assumptions}\label{section3}

\noindent \textbf{Regularity of the birth function:}  Let us recall some definition and properties of the so called bounded variation functions. Let $F: [0, \infty)\to\mathbb{R}$ be some function. For each $a>0$ define
$$ 
T_F(a):=\sup\big\{\sum_{j=1}^n|F(a_j)-F(a_{j-1})|: n\in\mathbb{N}, 0=a_0<\cdots<a_n=a\big\}\in[0, \infty], 
$$
where the supremum is taken over all finite strictly increasing sequences in $[0, a]$. 
\begin{definition}
	A function $F: [0, \infty)\to\mathbb{R}$ is said to be of bounded variation on $[0, \infty)$ if 
	$$ 
	\sup\limits_{a>0}T_F(a)<\infty. 
	$$
	In that case the function $a\to T_F(a)$ is bounded and increasing on $[0, \infty)$.
\end{definition}
Let $F: [0, \infty)\to\mathbb{R}$ be a right continuous function of bounded variation on $[0, \infty)$, then $a\to T_F(a)$ is also right continuous and according to Folland \cite[Theorem 3.29]{Folland1999}, there exists a unique Borel measure $\mu_F$ such that
$$ 
\mu_F((a, b])=F(b)-F(a),\text{ for all } a, b \in [0, \infty), \text{ with } a<b. 
$$
Furthermore, its total variation $|\mu_F|$ is the positive and bounded Borel measure associated to the right continuous and increasing function $a\to T_F(a)$. 

Next let us recall the integration by parts formula proved by Folland \cite[Theorem 3.36]{Folland1999} as well. If $G: [0, \infty)\to\mathbb{R}$ is of class $C^1$ then for all $0\le a<b$, one has
\begin{equation} \label{3.1}
	\int_a^bG(s)\mu_F(ds)=[G(b)F(b-)-G(a)F(a+)]-\int_a^bG'(s)F(s)ds.
\end{equation}

We now make a set of assumptions on the fertility rate $b$.
\begin{assumption}\label{ASS3.2}
        Assume that the function $\chi(a):=b(a) e^{-m \, a}$ satisfies the two following properties  
		\begin{itemize}
			\item [{\rm (i)}] 	Assume that $\chi\in L^1(0, \infty)$ with 
			$$
			\int_0^\infty \chi (a)da=1 \Leftrightarrow 	\int_0^\infty b(a) e^{-m\,a} da=1 . 
			$$
			\item [{\rm (ii)}]  We assume that the function $\tau: a\to a\chi(a)$ is right continuous and of bounded variation on $[0, \infty)$. We let $\mu_\tau$ be the unique Borel measure associated to $\tau$.
		\end{itemize}	
\end{assumption}
\begin{remark}
Note that when $b\in L^\infty(0,\infty)$, for each integer $n \in \N$ one has 
	$$
a \to 	a^n \chi(a)=a^n b(a)e^{-m \, a} \in L^1((0,\infty),\mathbb{R}).
	$$ 
\end{remark}
\noindent \textbf{Equilibria: } Recall that $\begin{pmatrix}
0\\\bar{u}
\end{pmatrix}\in D(A)$ is an equilibrium of the semiflow $\{U_\nu(t)\}_{t\ge0}$ if and only if 
$$ 
\begin{pmatrix}
	0\\\bar{u}
\end{pmatrix} \in D(A)\quad\text{and}\quad A\begin{pmatrix}
0\\\bar{u}
\end{pmatrix} +
H \begin{pmatrix}
	\nu,  \begin{pmatrix}
		0\\
		\bar{u}
	\end{pmatrix}
\end{pmatrix} =0. 
$$
As a consequence a positive equilibrium is given by 
$$
 \bar{u}(a)=\overline{w}_\nu e^{-m\, a},\;a\geq 0 
$$
where $\overline{w}_\nu>0$ becomes a solution of the following equation 
$$ 
\overline{w}_\nu=f\left(\nu, \overline{w}_\nu \right). 
$$
Our next assumption is concerned with the existence of such equilibrium point and its regularity with respect to the parameter $\nu$.
\begin{assumption}\label{ASS3.4}
	We assume that there exists an open interval $I$ such that for each $\nu\in I$ there exists a constant solution $\overline{w}_\nu \in \R$ of the equation
	$$
	\overline{w}_\nu=f(\nu, \overline{w}_\nu). 
	$$
	We assume further that the map $\nu \to \overline{w}_\nu$ is continuously differentiable on the interval $I$. \\
	In the sequel we set 
$$ 
\bar{v}_\nu=\begin{pmatrix}
	0\\\bar{u}_\nu
\end{pmatrix}\quad\text{with}\quad\bar{u}_\nu(a)=\overline{w}_\nu e^{-m\,a},\quad \forall\nu\in I. 
$$
\end{assumption}
In the following we will use the notation $\mathcal{L}(Y, Z)$ to denote the space of the linear bounded operators from $Y$ to $Z$ where $Y$ and $Z$ are two Banach spaces. Define for $\nu\in I$ the linear operator $ B_\nu: D(B_\nu)\subset X\to X$ as follows,
\begin{equation*}
D(B_\nu)=D(A)\text{ and } B_\nu x=Ax+\partial_v H(\nu,\bar{v}_\nu)x, \;\forall x\in D(B_\nu),
\end{equation*} 
wherein $\partial_v$ corresponds to the partial derivative of $H(\nu,v)$ with respect to $v$. The bounded linear operator $\partial_v H(\nu,\bar{v}_\nu) \in \mathcal{L}( X_0 , X) $ is defined by 
\begin{equation*}
	\partial_v H(\nu,\bar{v}_\nu) 
	\begin{pmatrix}
		0_\R\\
	\varphi 
	\end{pmatrix}=	\begin{pmatrix}
 \partial_w f\left(\nu, \overline{w}_\nu\right) \int_0^\infty b(l) \varphi(l) dl	\\
0_{L^1}
\end{pmatrix},\;\forall \begin{pmatrix}
		0_\R\\
	\varphi 
	\end{pmatrix}\in X_0.
\end{equation*}
Herein $\partial_w$ denotes the partial derivative of $f=f(\nu, w)$ with respect to $w$. 

\medskip 
 By using the result of Ducrot, Liu and Magal \cite{DLM2008}, the essential growth rate of the semigroup generated by $(B_\nu)_0$, the part of $B_\nu$ in the closure of its domain, satisfies
$$ 
\omega_{0, ess}((B_\nu)_0)\leq-\mu<0. 
$$ 
The following result follows from \cite[Theorem 4.3.27, Lemma 4.4.2,  Theorem 4.4.3-(ii)]{magal2018theory} to which we refer the reader for a proof and more details.  
\begin{lemma}The spectrum of $B_\nu$ in the half plane 
$$
\Omega:=\left\{ \lambda \in \C: \Re \lambda>-m\right\}
$$
contains only isolated eigenvalues which are poles of the resolvent of $B_\nu$.  
\end{lemma}

Recall that the characteristic function, describing the spectrum of $B_\nu$ in $\Omega$, is obtained by computing the resolvent of $B_\nu$ as presented in Liu, Magal and Ruan \cite{LMR11}.  We define the \textbf{characteristic function} for $\nu\in I$ and $\lambda\in\Omega$ as follows
$$
\Delta(\nu, \lambda):= 1-  \partial_w f\left(\nu, \overline{w}_\nu\right) \int_0^\infty b(l)e^{-\left(m + \lambda \right)l}dl. 
$$

\medskip

Recall that the resolvent set $\rho(A)$ of $A$ contains $ \Omega $ and for each $\lambda \in \Omega$ the resolvent of $A$ is defined by the following formula 
\begin{equation*}
	\begin{array}{l}
	\left( \lambda I- A \right)^{-1} \left( \begin{array}{c}
		\alpha \\
		\psi
	\end{array}\right)
	=
	\left( \begin{array}{c}
		0 \\
		\varphi 
	\end{array}\right) 	\\
\\
\Leftrightarrow 	\varphi (a)=e^{- \left(\lambda+m\right) a } \alpha + \int_0^a e^{- \left(\lambda+m\right) \left(a-s\right)} \psi(s) ds.
	\end{array}
\end{equation*}
We now recall some result already presented (in a more general framework in the Section 5.2 in \cite{LMR11}) 
\begin{lemma} For each $\nu\in I$ the  resolvent set $\rho(B_\nu)$   of $B_\nu$  satisfies 
	\begin{equation*}
\lambda \in	\rho(B_\nu)  \cap \Omega  \Leftrightarrow \Delta(\nu, \lambda) \neq 0,
	\end{equation*}
or equivalently  the spectrum $\sigma(B_\nu):= \C \setminus \rho(B_\nu)$   of $B_\nu$  satisfies 
	\begin{equation*}
	\lambda \in \sigma(B_\nu)  \cap \Omega  \Leftrightarrow \Delta(\nu, \lambda)=0. 
\end{equation*} 
Moreover one has
\begin{equation*}
	\begin{array}{l}
		\left( \lambda I- B_\nu \right)^{-1} \left( \begin{array}{c}
			\alpha \\
			\psi
		\end{array}\right)
		=
		\left( \begin{array}{c}
			0 \\
			\varphi 
		\end{array}\right) 	\\
		\\
		\Leftrightarrow 	\varphi (a)=e^{- \left(\lambda+m\right) a } \alpha_1 + \int_0^a e^{- \left(\lambda+m\right) \left(a-s\right)} \psi(s) ds,
	\end{array}
\end{equation*}
where
$$
 \alpha_1 :=\Delta(\nu, \lambda)^{-1} \left[ \alpha  -  \partial_w f\left(\nu, \overline{w}_\nu\right) \int_0^\infty b(a)  \int_0^a e^{- \left(\lambda+m\right) \left(a-s\right)}	\psi(s) dsda\right].
$$
\end{lemma}

\begin{proof} For each $\lambda \in \Omega$ the linear operator $\lambda I -B_\nu$ is invertible if and only if the linear operator is invertible 
	$$
	I - \partial_v H(\nu,\bar{v}_\nu) 	\left( \lambda I- A \right)^{-1} . 
	$$
	In that case we have 
	$$
	(\lambda I -B_\nu)^{-1}=	\left( \lambda I- A \right)^{-1}  \left[ 	I - \partial_v H(\nu,\bar{v}_\nu) 	\left( \lambda I- A \right)^{-1} \right]^{-1}. 
	$$
\noindent \textbf{Computation of the inverse of $ 	  	I - \partial_v H(\nu,\bar{v}_\nu) 	\left( \lambda I- A \right)^{-1} $	:} We have 
$$
\begin{array}{l}
	\left( \begin{array}{c}
	\widehat{\alpha } \\
	\widehat{	\varphi }
\end{array}\right) =\left[I - \partial_v H(\nu,\bar{v}_\nu) 	\left( \lambda I- A \right)^{-1}  \right] 
\left( \begin{array}{c}
	\alpha  \\
	\varphi
\end{array}\right) 
\\
\Leftrightarrow 
	\left( \begin{array}{c}
	\widehat{\alpha } \\
	\widehat{	\varphi }
\end{array}\right) =
\left( \begin{array}{c}
	\alpha-  \partial_w f\left(\nu, \overline{w}_\nu\right) \int_0^\infty b(a)\left[ e^{- \left(\lambda+m \right) a } \alpha + \int_0^a e^{- \left(\lambda+m \right) \left(a-s\right)} \	\varphi(s) ds \right] da\\
	\varphi
\end{array}\right). 
\end{array} 
$$
Therefore $ 	  	I - \partial_v H(\nu,\bar{v}_\nu) 	\left( \lambda I- A \right)^{-1} $ is invertible (for $\lambda \in \Omega$)  if and only if $\Delta(\nu, \lambda) \neq 0$ and one has
$$
\begin{array}{l}
		\left( \begin{array}{c}
		\alpha  \\
		\varphi
	\end{array}\right) 
=\left[I - \partial_v H(\nu,\bar{v}_\nu) 	\left( \lambda I- A \right)^{-1}  \right]^{-1} 
	\left( \begin{array}{c}
	\widehat{\alpha } \\
	\widehat{	\varphi }
\end{array}\right) 
	\\
	\Leftrightarrow 
	\left( \begin{array}{c}
	\alpha  \\
		\varphi 
	\end{array}\right) =
	\left( \begin{array}{c}
\Delta(\nu, \lambda)^{-1} \left[	\widehat{\alpha } -  \partial_w f\left(\nu, \overline{w}_\nu\right) \int_0^\infty b(a)  \int_0^a e^{- \left(\lambda+m \right) \left(a-s\right)}\widehat{	\varphi }(s) dsda\right]	\\
	\widehat{	\varphi }
	\end{array}\right).
\end{array} 
$$
The result follows. 
\end{proof}

\begin{assumption}\label{ASS3.7} There exists $\nu_0 \in I$ and $\omega_0>0$ such that the following properties are satisfied
	\begin{itemize}
			\item [{\rm (i)}] 
		$$
	\Delta(\nu_0 , \omega_0 i )=0.
		$$

		\item [{\rm (ii)}] Crandall and Rabinowitz's condition
		$$
	\Delta(\nu_0 , k\,   \omega_0 i )\neq 0,  \forall k \in \N \text{ and } k \neq 1. 
		$$
		
			\item [{\rm (iii)}] Simplicity of the eigenvalue $ \omega_0 i$
		$$ 
	\partial_\lambda	\Delta(\nu_0 ,    \omega_0 i )\neq 0, 
		$$
	which is also equivalent to
	$$
	\partial_w f\left(\nu_0, \overline{w}_{\nu_0}\right) \int_0^\infty b(l) \, l \, e^{- \left( m+  \omega_0 i  \right)l}dl \neq 0. 
	$$
			\item [{\rm (iv)}]  Transversality condition
		$$
	\Re \left(	\partial_\nu	\Delta(\nu_0 ,    \omega_0 i )\times  \overline{ \partial_\lambda	\Delta(\nu_0 ,    \omega_0 i ) }  \right) \neq 0.
		$$
	
	\end{itemize}
\end{assumption}
We can observe that by combining (i) and (iii)  and by using the implicit function theorem there exists a branch $\lambda: (\nu_0-\eta, \nu_0+\eta)\subset I\to\mathbb{C}$ with some $\eta>0$ small enough such that for each $\nu\in(\nu_0-\eta, \nu_0+\eta)$, $\lambda(\nu)=\alpha(\nu)+i \omega(\nu)$ and $\overline{\lambda(\nu)}=\alpha(\nu)-i \omega(\nu)$ satisfying  solution of 
\begin{equation} \label{3.2}
\Delta(\nu , \lambda(\nu))=0	
\end{equation}
and 
$$
\lambda(\nu_0)=i\omega_0.
$$ 
Moreover by combining (iii)-(iv) we deduce that the transversality condition is satisfied. Namely, one has  
$$ 
\Re \frac{d\lambda(\nu_0)}{d\nu}\neq0. 
$$
Moreover by using the property (iii) in Lemma 5.8 in \cite{LMR11} we also deduce that $i \omega_0$ is a simple eigenvalue of $B_{\nu_0}$  since
$$
\lim_{\lambda \to \lambda_0} \dfrac{\Delta(\nu_0, \lambda)}{\lambda-\lambda_0} \neq 0 \Leftrightarrow \partial_\lambda \Delta(\nu_0, \lambda_0)\neq 0. \quad(\text{with } \lambda_0=\lambda(\nu_0) ).
$$
The condition (ii) avoid to assume that the purely imaginary spectrum is reduced to the a single pair of purely imaginary eigenvalues. Such a condition has been introduced in the Crandall and Rabinowitz's proof \cite{crandall1977hopf}.

\section{Main Result}\label{section4}

In this section we state the main theorem of this paper. The following result is inspired by Crandall and Rabinowitz \cite{crandall1977hopf}.
\begin{theorem}\label{TH4.1}
	Let Assumptions \ref{ASS3.2},  \ref{ASS3.4} and \ref{ASS3.7} be satisfied. Then there exist a constant $\delta>0$ and two $C^1$ maps, $s\to\nu(s)$ from $(-\delta, \delta)$ to $\mathbb{R}$ and $s\to\omega(s)$ from $(-\delta, \delta)$ to $\mathbb{R}$ such that for each $s\in(-\delta, \delta)$ there exists a $2\pi/\omega(s)-$periodic solution $u(s)$ of class of $C^1$ which is a solution of \eqref{1.1} with the parameter $\nu=\nu(s)$. Moreover, the branch of periodic orbit is bifurcating from $\nu_0$ at $\nu=\nu_0$, that is to say that $$ \nu(0)=\nu_0, \quad\omega(0)=\omega_0. $$ 
\end{theorem} 

\begin{proof}(Proof of Theorem \ref{TH4.1})
	Up to time rescaling we can assume, without loss of generality, that $\omega_0=1$. Observe that Assumption \ref{ASS3.7}-(i) implies that \eqref{1.1} linearized about $u=\bar{u}_{\nu_0}$ for $\nu=\nu_0$ has nontrivial $2\pi-$periodic solutions. We now seek nontrivial $2\pi/\omega-$periodic solutions of \eqref{1.1} with $\omega$ close to $1$ and $(u, \nu)$ close to $(\bar{u}_{\nu_0}, \nu_0)$. 
	
	Solving \eqref{1.1} along the characteristic line $t-a=$constant, one obtains
	$$
	u(t,a)=u(t-a,0)e^{-m \, a},\;t\in \R ,a>0.
	$$
	Thus $v=v(t)$ given by
	\begin{equation*}
	v(t)=u(t,0),
	\end{equation*}
	satisfies the renewal equation
	\begin{equation*}
	v(t)=f\left(\nu, \int_0^\infty b(a) v(t-a) e^{-m\, a}da\right),\;t\in\R.
	\end{equation*}
	Setting
	$$
	w(t)=v\left(\frac{t}{\omega}\right),
	$$
	yields the following equation for the $2\pi-$periodic function $w=w(t)$
	\begin{equation*}
	w(t)=v\left(\frac{t}{\omega}\right)=f\left(\nu, \int_0^\infty b(a) v(t/\omega-a) e^{-m \, a}da\right)=f\left(\nu, \int_0^\infty b(a) w(t-\omega a) e^{-m\, a}da\right),\;t\in\R.
	\end{equation*}
	Recalling the definition of $\chi$ in Assumption \ref{ASS3.2} and using the change of the variable $l=\omega a$ in the integral lead to the following equation for $w=w(t)$
	\begin{equation}\label{4.1}
	w(t)=f\left(\nu, \int_0^\infty\frac{1}{\omega}\chi\left(\frac{l}{\omega}\right) w(t-l)dl\right),\;t\in\R.
	\end{equation}
	Now the existence of nontrivial $2\pi/\omega-$periodic solution of \eqref{1.1} becomes equivalent to the one of nontrivial $2\pi-$ periodic solution of \eqref{4.1}. Next we shall apply the implicit function theorem to investigate the existence of nontrivial $2\pi-$periodic solution of \eqref{4.1} for $\nu$ close to $\nu_0$. 
	
	Let $C_{2\pi}(\mathbb{R})$ be the Banach space of the continuous $2\pi-$periodic functions. Define the map $F: \mathbb{R}^2\times C_{2\pi}(\mathbb{R})\to C_{2\pi}(\mathbb{R})$ by
	$$ 
	F(\omega, \nu, x)(t)=x(t)-f\left(\nu, \int_0^\infty\frac{1}{\omega}\chi\left(\frac{l}{\omega}\right) x(t-l)dl\right),\quad \forall(\omega, \nu, x)\in\mathbb{R}^2\times C_{2\pi}(\mathbb{R}). 
	$$
	We now aim at investigating the zeros of the equation 
	$$ 
	F(\omega, \nu, x)=0, 
	$$
	for $(\omega, \nu, x)$ close to $(1, \nu_0, \overline{w}_{\nu_0})$ using the implicit function theorem. To do so, we need first to verify the smoothness of  
	$$ \int_0^\infty\frac{1}{\omega}\chi\left(\frac{l}{\omega}\right)x(t-l)dl=\int_0^\infty \chi(l) x(t-\omega l)dl, 
	$$
	with respect to $\omega$. 
	
    The first main step of this proof is the following lemma. 
	\begin{lemma}\label{LE4.2} The map $G:\R\times C_{2\pi}(\mathbb R)\to C_{2\pi}(\mathbb R)$ defined by
\begin{equation*}
G(\omega,x)(t)=\int_0^\infty \chi(l) x(t-\omega l)dl,\;t\in\R,
\end{equation*}	
is continuously differentiable with respect to $\omega\in\R$ and its partial derivative with respect to $\omega$, denoted by $\partial_\omega G$, is given 
\begin{equation*}
\partial_\omega G(\omega,x)(\cdot)=\int_0^\infty x(\cdot-\omega l)\mu_\tau(dl),\;\forall (\omega,x)\in\R\times C_{2\pi}(\mathbb R).
\end{equation*}
Herein $\tau$ and $\mu_\tau$ are defined in Assumption \ref{ASS3.2}-(ii).
	\end{lemma}

	\begin{proof}
 
		First observe that using Fubini theorem, for any $x\in C^1(\mathbb{R}) \cap C_{2\pi}(\mathbb{R})$ and $\omega\in\R$ we have, for any $t\in\R$, 
		\begin{equation*} 
		\begin{array}{ll}
		\int_0^\omega \int_0^\infty - x'(t-\sigma l) \,l \,\chi(l) \, dl \, d\sigma &=\int_0^\infty	\int_0^\omega -l x'(t-\sigma l)  \, d\sigma \, \chi(l) \,dl  \\
		&=\int_0^\infty	\left[ x(t-\sigma l) \right]_{\sigma=0}^{\sigma=\omega} \,  \chi(l)  \, dl \\
		&= \int_0^\infty	\chi(l) x(t-\omega l)  dl  - x(t)   \int_0^\infty	 \chi(l)dl,
		\end{array}	
		\end{equation*}
		so that, since $\int_0^\infty \chi(l)dl=1$ (see Assumption \ref{ASS3.2}-(i)), we get
	\begin{equation} \label{4.2}
		\int_0^\omega \int_0^\infty - x'(t-\sigma l)l\chi(l)dl d\sigma= \int_0^\infty \chi(l) x(t-\omega l)  dl  - x(t),	
	\end{equation}
	that is for all $(\omega,x)\in \R\times\left(C^1(\mathbb{R}) \cap C_{2\pi}(\mathbb{R})\right)$ and $t\in\R$
\begin{equation*}
G(\omega,x)(t)=x(t)-\int_0^\omega \int_0^\infty x'(t-\sigma l)\tau(l)dl d\sigma.
\end{equation*}	
We deduce that for all $x\in C^1(\mathbb{R}) \cap C_{2\pi}(\mathbb{R})$ the map  $\omega \to  G(\omega,x)=\int_0^\infty	\chi(l) x(\cdot-\omega l)  dl $ is of class $C^1$ and 
			\begin{equation} \label{4.3}
			\frac{d}{d\omega}\int_0^\infty \chi(l) x(t-\omega l) dl=-\int_0^\infty x'(t-\omega l)l \chi(l) dl.
			\end{equation} 
    Moreover by using again the formula \eqref{4.2} we get
	\begin{equation*} 
	\int_{\omega}^{\omega+\varepsilon} \int_0^\infty - x'(t-\sigma l)l\chi(l)dl d\sigma= \int_0^\infty	\chi(l) x(t-(\omega + \varepsilon)l)  dl  -  \int_0^\infty \chi(l) x(t-\omega l) dl  . 
    \end{equation*} 
    hence 
		\begin{equation} \label{4.4}
			\int_{\omega}^{\omega+\varepsilon}\frac{d}{d\sigma}\int_0^\infty \chi(l) x(t-\sigma l)dl d\sigma= \int_0^\infty \chi(l)	 x(t-(\omega + \varepsilon)l) dl  -  \int_0^\infty	\chi(l) x(t-\omega l) dl  . 
			\end{equation} 
    By using the integration by parts formula \eqref{3.1} and \eqref{4.3}, we obtain  for all $(\omega,x)\in\R\times\left(C^1(\mathbb{R}) \cap C_{2\pi}(\mathbb{R})\right)$
		\begin{equation*}
		\partial_\omega G(\omega,x)(t)=\frac{d}{d\omega}\int_0^\infty \chi(l) x(t-\omega l)dl=-\int_0^\infty x'(t-\omega l)l\chi(l)dl=\int_0^\infty x(t-\omega l)\mu_\tau(dl).
		\end{equation*} 
		Then using Assumption \ref{ASS3.2}-(ii) we infer from the above equality that
		\begin{equation*}
		\norm{\partial_\omega G(\omega,x)}_{C_{2\pi}(\mathbb{R})}\le \norm{x}_{C_{2\pi}(\mathbb{R})}|\mu_\tau|((0, \infty))<\infty,\, \forall (\omega,x)\in \R\times C^1(\mathbb{R}) \cap C_{2\pi}(\mathbb{R}),
		\end{equation*} 
    wherein $|\mu_\tau|((0, \infty))$ is nothing but the variation of $\tau(a)$ on $(0, \infty)$. That is supremum  over all the subdivision $0=t_0 <t_1 <t_2 <\ldots <t_n =M$ of 
    $$
    |\mu_\tau|((0,M])=\sup_{0=t_0 <t_1 <t_2 <\ldots <t_n =M} \sum_{i=0}^{n} \vert \tau(t_{i+1})-\tau(t_i) \vert, 
    $$
    and 
    $$
    |\mu_\tau|((0,\infty))=\lim_{M \to +\infty} |\mu_\tau|((0,M]). 
    $$
    We now define $L_\omega \in \mathcal{L} \left( C_{2\pi}(\mathbb{R})\right) $ by
		$$ 
		L_\omega(x)(t):=\int_0^\infty x(t-\omega l)\mu_\tau(dl). 
		$$	
    Now the result follows by using the \eqref{4.4} and the density of $C^1(\mathbb{R}) \cap C_{2\pi}(\mathbb{R})$ into $C_{2\pi}(\mathbb{R})$, since it implies that 
    \begin{equation*}
	\int_{\omega}^{\omega+\varepsilon} L_\sigma (x)(t) d\sigma= \int_0^\infty	\chi(l) x(t-(\omega + \varepsilon)l)  dl  -  \int_0^\infty	\chi(l) x(t-\omega l)  dl. 
    \end{equation*}
    It remains to prove that the map $(\omega, x) \to	\partial_\omega G(\omega,x)=\int_0^\infty x(\cdot-\omega l)\mu_\tau(dl)$ is continuous from $\mathbb{R}\times C_{2\pi}(\mathbb{R})$ into $C_{2\pi}(\mathbb{R})$. To see this fix $(\omega_1,x_1)$ in $\mathbb{R}\times C_{2\pi}(\mathbb{R})$ and observe that for all $(\omega,x)\in \mathbb{R}\times C_{2\pi}(\mathbb{R})$ one has:
    \begin{equation*}
	\begin{array}{l}
		\partial_\omega G(\omega, x)-\partial_\omega G(\omega_1, x_1)
		=\int_0^\infty\left[x(\cdot-\omega l)-x_1(\cdot-\omega_1 l)\right] \mu_\tau(dl)={\rm J_1+J_2},
		\end{array}
    \end{equation*}
		wherein we have set
		\begin{equation*}
		{\rm J_1}:=\int_0^\infty\left[x(\cdot-\omega l)-x_1(\cdot-\omega l)\right] \mu_\tau(dl),\;{\rm J_2}:=\int_0^\infty\left[x_1(\cdot-\omega l)-x_1(\cdot-\omega_1l)\right]\mu_\tau(dl).
    \end{equation*}
    We first observe that 
    \begin{equation*}
	\|{\rm J_1}\|_{C_{2\pi}(\mathbb{R})}\le\|x-x_1\|_{C_{2\pi}(\mathbb{R})}|\mu_\tau|((0, \infty))\to 0\text{ uniformy for $\omega\in\R$ as $\|x-x_1\|_{C_{2\pi}(\mathbb{R})}\to 0$}.
    \end{equation*} 
    On the other hand fix $\varepsilon>0$ and since one has
    $$
    |\mu_\tau|((M, \infty)):=|\mu_\tau|((0, \infty))-|\mu_\tau|((0, M]) \to 0 \text{ as } M \to \infty,
    $$
    choose $M>0$ large enough so that $2 \|x_1\|_{C_{2\pi}(\R)}  |\mu_\tau|((M, \infty))\leq\varepsilon$.
    With such a choice we split ${\rm J_2}$ as follows ${\rm J_2=I_1+I_2}$ with 
    \begin{equation*}
	{\rm I_1}:=\int_0^M\left[x_1(\cdot-\omega l)-x_1(\cdot-\omega_1 l)\right]\mu_\tau(dl),\;{\rm I_2}:=\int_M^\infty\left[x_1(\cdot-\omega l)-x_1(\cdot-\omega_1 l)\right]\mu_\tau(dl).
    \end{equation*}
    Hence we get
    \begin{equation*}
	\|{\rm I_2}\|_{C_{2\pi}(\R)}\leq 2 \|x_1\|_{C_{2\pi}(\R)}  |\mu_\tau|((M, \infty))\leq \varepsilon,
    \end{equation*}
    and, since $x_1$ is continuous and $2\pi$-periodic, it is uniformly continuous on $\R$. Thus for all $\omega$ is sufficiently close to $\omega_1$, we have 
    $$ 
    \|{\rm I_1}\|_{C_{2\pi}(\R)}\le  \sup_{0\leq l \leq M} \vert x_1(t-\omega l)-x_1(t-\omega_1l)  \vert |\mu_\tau|((0, M])\leq\varepsilon. 
    $$
    As a consequence $\partial_\omega G$ is continuous and the proof is completed.    
    \end{proof}
\noindent \textbf{Computation of the derivatives of $F$: } One can calculate the following derivatives directly,
	\begin{equation}\label{4.5}
\partial_x	F (\omega, \nu, \overline{w}_{\nu})(x)(t) = x(t) - \partial_w f (\nu, \overline{w}_{\nu}) \int_0^\infty \chi(l) x(t-\omega l)dl, 
	\end{equation}
	\begin{equation}\label{4.6}
\partial_\nu  \partial_x F (\omega, \nu, \overline{w}_{\nu}) (x) (t) = -\left[ \partial_\nu \partial_w f(\nu, \overline{w}_{\nu})+\partial_w ^2 f (\nu, \overline{w}_{\nu}) \partial_\nu \overline{w}_{\nu} \right] \int_0^\infty \chi(l) x(t-\omega l) dl, 
	\end{equation}
and by Lemma \ref{LE4.2} 
\begin{equation}\label{4.7}
	\partial_\omega \partial_x 	F (\omega, \nu, \overline{w}_{\nu}) (x) (t) = -\partial_w f (\nu, \overline{w}_{\nu}) \int_0^\infty x(t-\omega l)\mu_\tau(dl). 
\end{equation}
\noindent \textbf{State space decomposition:} Note that by Assumption \ref{ASS3.7}-(i) we have 
\begin{equation*}
	1=\partial_w f (\nu_0, \overline{w}_{\nu_0}) \int_0^\infty\chi(l)e^{\pm i\,l}dl,
	\end{equation*}
which re-writes as
\begin{equation}\label{4.8}
    1=\partial_w f (\nu_0, \overline{w}_{\nu_0}) \int_0^\infty\chi(l) \cos l dl \text{ and }0=\partial_w f (\nu_0, \overline{w}_{\nu_0}) \int_0^\infty\chi(l) \sin l dl,
\end{equation}
and thus the kernel $N(\partial_ x F (1, \nu_0, \overline{w}_{\nu_0}))$ contains the space $X_1$ given by
	\begin{equation*} 
	X_1: = \text{span} \{\cos,\sin\}\subset C_{2\pi}(\mathbb{R}). 		
	\end{equation*} 
	Now define the closed space 
\begin{equation*}
	X_2:=\bigg\{z\in C_{2\pi}(\mathbb{R}): \int_0^{2\pi}z(t)e^{i\,t}dt=0  \bigg\}=\bigg\{z\in C_{2\pi}(\mathbb{R}): \int_0^{2\pi}z(t)\cos(t)dt=\int_0^{2\pi}z(t)\sin(t)dt=0  \bigg\}. 	
\end{equation*}
This space turns out to be a complement of $X_1$ as stated in the next lemma.
\begin{lemma}\label{LE4.3} We have the following state space decomposition 
	\begin{equation*}
	 C_{2\pi}(\mathbb{R}) = X_1 \oplus X_2.
	 \end{equation*}
 \end{lemma}

\begin{proof}
This property is directly inherited from the decomposition of $L^2((0, 2\pi), \mathbb{R})$ as
\begin{equation*}
	L^2((0, 2\pi), \mathbb{R})=X_1\oplus X_1^\perp,
\end{equation*} 
with 
\begin{equation*} 
	 X_1^\perp=\bigg\{z\in L^2((0, 2\pi), \mathbb{R}): \int_0^{2\pi}z(t)e^{i\,t}dt=0  \bigg\},
\end{equation*}
Now if $z \in C_{2\pi}(\mathbb{R})$ then $z \in L^2((0, 2\pi), \mathbb{R})$ and the above $L^2((0, 2\pi), \mathbb{R})-$decomposition ensures that there exist unique $z_1 \in X_1$ and $z_2 \in X_1^\perp$ such that 
	$$
	z=z_1+z_2.
	$$
Now since $z_1 =c_1 \cos + c_2 \sin$ (for some constants $c_1$ and $c_2$) this ensures that $z_2=z-z_1$ is also continuous and $z_2 \in X_2$. The state space decomposition follows.
\end{proof}

	Now let us define the map $h:\mathbb{R}^3\times X_2\to C_{2\pi}(\mathbb{R})$ by
	\begin{equation*}
	h(s, \omega, \nu, z)=\begin{cases}
	s^{-1} F (\omega, \nu, \overline{w}_\nu+s(u_1+z)),\quad &\text{if}\,s\neq0,\\
	\partial_x F (\omega, \nu, \overline{w}_\nu) (u_1+z),\quad &\text{if}\, s=0,
	\end{cases}
	\end{equation*}
	where 
	$$
	u_1(t)=\cos(t),\;\forall t\in\R.
	$$ 
	Now let us observe that since $f=f(\nu,x)$ is of class $C^2$, $h$ is of class $C^1$. One also has $h(0, 1, \nu_0, 0)=0$ while the derivative with respect to $(\omega, \nu, z)$ is given, for all $(\tilde{\omega}, \tilde{\nu}, \tilde{z})\in\R\times\R\times C_{2\pi}(\mathbb{R})$, by
	\begin{equation*}
		\begin{array}{l}
		D_{(\omega, \nu, z)} h(0, 1, \nu_0, 0) (\tilde{\omega}, \tilde{\nu}, \tilde{z}) 
		
		=\partial_x F (1, \nu_0, \overline{w}_{\nu_0}) \tilde{z}  + \tilde{\nu}  \partial_\nu  \partial_x F (1, \nu_0, \overline{w}_{\nu_0}) u_1+ \tilde{\omega} \partial_\omega  \partial_x F (1, \nu_0, \overline{w}_{\nu_0}) u_1,

		\end{array}
	\end{equation*}
hence by using \eqref{4.5}-\eqref{4.7} we obtain 
	\begin{equation*}
	\begin{array}{ll}
			D_{(\omega, \nu, z)} h(0, 1, \nu_0, 0) (\tilde{\omega}, \tilde{\nu}, \tilde{z}) =
	&	\tilde{z}(t) - \partial_w f (\nu_0, \overline{w}_{\nu_0}) \int_0^\infty\chi(l)\tilde{z}(t-l)dl \vspace{0.2cm} \\
		
	&	-\tilde{\nu}\left[ \partial_\nu \partial_w f (\nu_0, \overline{w}_{\nu_0}) + \partial_w ^2 f (\nu_0, \overline{w}_{\nu_0}) \partial_\nu \overline{w}_{\nu_0} \right]\int_0^\infty\chi(l)u_1(t-l)dl  \vspace{0.2cm}\\
	
		&	- \tilde{\omega} \partial_w f(\nu_0, \overline{w}_{\nu_0}) \int_0^\infty u_1(t-l)\mu_\tau(dl) .
	\end{array}
\end{equation*}
The second main step of the proof is the following lemma. 	
	\begin{lemma} \label{LE4.4}
		The bounded linear operator 
		$$ 
		D_{(\omega, \nu, z)}h(0, 1, \nu_0, 0)\in\mathcal{L}(\mathbb{R}^2\times X_2, C_{2\pi}(\mathbb{R})) 
		$$
		is invertible. 
	\end{lemma}

	\begin{proof}
		Let us first define the linear bounded operator $K: X_2\to C_{2\pi}(\mathbb{R})$ by 
		$$ 
		K\phi:=\partial_w f (\nu_0, \overline{w}_{\nu_0}) \int_0^\infty\chi(l)\phi(\cdot-l)dl. 
		$$
\noindent \textbf{Step 1: Let us prove  that $K(X_2) \subset X_2$.}	 Indeed, by using Fubini's theorem, for all $\phi\in C_{2\pi}(\mathbb{R})$ one has 
		$$
		\int_0^{2\pi}\int_0^\infty\chi(l)\phi(t-l)dl e^{i\,t}dt= \int_0^\infty	\int_0^{2\pi} \phi(t-l) e^{i\,(t-l)}dt  e^{i \, l} \chi(l) dl
		$$
		and since $t \to \phi(t) e^{i \, t}$ is $2 \pi-$periodic we deduce that  
		$$
		\int_0^{2\pi} \phi(t-l) e^{i\,(t-l)}dt =0, \forall l \geq 0.  
		$$
		This completes the first step.
		
\noindent \textbf{Step 2: Let us now prove that $N(I-K)=\{0\}$ whenever $I-K \in \mathcal{L}(X_2)$.} In order to compute the kernel of $I-K$ in $X_2$, consider $g\in N(I-K)$, that is $g\in X_2$ such that  
		
		\begin{equation*}
		g(t)-\partial_w f (\nu_0, \overline{w}_{\nu_0}) \int_0^\infty\chi(l)g(t-l)dl=0,\;\forall t\in \R.
		\end{equation*}		
	Multiplying the above equality by $ e^{-int}$, for some $n\in \mathbb Z$ and integrating between $0$ and $2 \pi$ we obtain   	
		\begin{equation*}
		\left[1-\partial_w f (\nu_0, \overline{w}_{\nu_0}) \widehat{\chi}(n)\right]\widehat{g}(n)=0,
	\end{equation*}
	wherein we have set
	$$ 
	\widehat{g}(n) := \int_{0}^{2\pi} g(l) e^{-inl} dl,
	$$
	and 
\begin{equation*}
	\widehat{\chi}(n) : = \int_0^\infty \chi(l) e^{-inl} dl. 
\end{equation*}
Now for $n\neq\pm 1$ we deduce by Assumption \ref{ASS3.7}-(ii) that 
	$$ 
	\widehat{g}(n) = \left[1-\partial_w f (\nu_0, \overline{w}_{\nu_0}) \widehat{\chi}(n)\right]^{-1}0=0. 
	$$
Since $g\in X_2$, it follows that $g=0$ and $N(I-K)=\{0\}$. 

\medskip 
\noindent \textbf{Step 3: Let us prove that $I-K \in \mathcal{L}(X_2) $ is invertible.} Next note that it follows from the continuity of the translation in $L^1$ that $K$ is a compact operator. Thus one has $R(I-K)=X_2$ by Fredholm Alternative (see \cite[Lemma 4.3.17]{magal2018theory}), where $R(I-K)$ denotes the range of $I-K$. Hence we have that $I-K$ is invertible and that the inverse is continuous by bounded inverse theorem. 
		
\medskip 		
\noindent \textbf{Step 4: Let us prove that $D_{(\omega, \nu, z)}h(0, 1, \nu_0, 0)$ is invertible.}			To prove this, let $y\in C_{2\pi}(\mathbb{R})$ be given. Set $L:=D_{(\omega, \nu, z)}h(0, 1, \nu_0, 0)$ and let us consider the equation 
\begin{equation}\label{4.9}
\left(\tilde{\omega}, \tilde{\nu}, \tilde{z}\right)\in \R\times\R\times X_2,\;\;	L(\tilde{\omega}, \tilde{\nu}, \tilde{z})=y.
\end{equation} 
Define the projectors $P_1: C_{2\pi}(\mathbb{R})\to X_1$ and $P_2: C_{2\pi}(\mathbb{R})\to X_2$ associated to the state space decomposition of Lemma \ref{LE4.3} and set $y_1:=P_1y$ and $y_2:=P_2y$. Next projecting \eqref{4.9} along $P_1$ and $P_2$ yields the following system, for all $t\in\R$,

\begin{equation}\label{4.10}
	\begin{array}{ll}
		y_1(t) =& -\tilde{\nu} \left[\partial_\nu \partial_w  f (\nu_0, \overline{w}_{\nu_0}) + \partial_w ^2 f (\nu_0, \overline{w}_{\nu_0}) \partial_\nu \overline{w} _{\nu_0} \right] \int_0^\infty \chi(l) u_1(t-l) dl  \vspace{0.1cm}\\
		&-\tilde{\omega} \partial_w f (\nu_0, \overline{w}_{\nu_0}) \int_0^\infty u_1(t-l) \mu_\tau(dl),
	\end{array}
\end{equation} 
and 
\begin{equation*}
	\begin{array}{ll}
		y_2(t) =& \tilde{z}_2(t)-\partial_w f (\nu_0, \overline{w}_{\nu_0}) \int_0^\infty \chi(l) \tilde{z}_2(t-l)dl.\nonumber
	\end{array}
\end{equation*} 
\medskip 
\noindent Observe that $I-K$ is invertible in $X_2$, thus $\tilde{z}_2$ can be solved by 
$$ 
\tilde{z}_2=(I-K)^{-1}y_2. 
$$

Let us focus on the resolution of \eqref{4.10}. To that aim, recall that $u_1=\cos(\cdot)$ and \eqref{4.8} so that \eqref{4.10} rewrites

		\begin{equation*}
		\begin{array}{ll}
		y_1(t) =& -\tilde{\nu} \left[\partial_\nu \partial_w  f (\nu_0, \overline{w}_{\nu_0}) + \partial_w ^2 f (\nu_0, \overline{w}_{\nu_0}) \partial_\nu \overline{w} _{\nu_0} \right] \int_0^\infty \chi(l) \cos l dl \cos t \vspace{0.1cm}\\
		&-\tilde{\omega} \partial_w f (\nu_0, \overline{w}_{\nu_0}) \left[ \int_0^\infty \cos l \mu_\tau(dl) \cos t + \int_0^\infty \sin l \mu_\tau(dl) \sin t \right],
		\end{array}
		\end{equation*} 
		Furthermore by applying the integration by parts formula \eqref{3.1},  we obtain    
		$$
		\begin{array}{l}
		 \int_0^\infty \cos l \mu_\tau (dl)=\lim_{M \to \infty}  \int_0^M \cos l \mu_\tau (dl)\\
		 = \lim_{M \to \infty} \left\{  \left[ \cos M \tau(M) - \cos 0 \tau(0) \right]+\int_0^M \tau(l) \sin l dl \right\},
		 \end{array}
		$$
		hence 
		$$
		 \int_0^\infty \cos(l)\, \mu_\tau (dl)= \int_0^\infty l \chi(l) \sin(l) dl.
		$$
		Similarly, one can obtain
		$$
		\int_0^\infty \sin l \mu_\tau (dl)= - \int_0^\infty l \chi(l) \cos l  dl.
		$$
		On the other hand, since $y_1\in X_1$, there exist two constants $c_1,c_2\in\R$ such that $y_1=c_1\cos(\cdot)+c_2\sin(\cdot)$, while  the $y_1$-equation can be rewritten as the following system, for all $t\in\R$,
		\begin{equation*}
		\begin{array}{ll}
	     &c_1 \cos t + c_2 \sin t \vspace{0.1cm}\\
	     =& -\tilde{\nu} \left[ \partial_\nu \partial_w f (\nu_0, \overline{w}_{\nu_0}) + \partial_w ^2 f (\nu_0, \overline{w}_{\nu_0}) \partial_\nu \overline{w}_{\nu_0} \right]  \int_0^\infty \chi(l) \cos l dl \cos t \vspace{0.1cm}\\ 
	     &- \tilde{\omega}  \partial_w  f  (\nu_0, \overline{w}_{\nu_0}) \left[ \int_0^\infty l \chi(l) \sin l dl \cos t- \int_0^\infty l  \chi(l) \cos l dl \sin t \right],
		\end{array}
		\end{equation*}
		and identifying the coefficients of $\cos(\cdot)$ and $\sin(\cdot)$ we end-up with the resolution of the following two-dimensional linear system
		\begin{equation*}
		\begin{pmatrix}-\left [ \partial_\nu \partial_w f (\nu_0, \overline{w}_{\nu_0}) + \partial_w ^2 f (\nu_0, \overline{w}_{\nu_0}) \partial_\nu \overline{w}_{\nu_0} \right] \int_0^\infty \chi(l) \cos l dl & -\partial_w f (\nu_0, \overline{w}_{\nu_0}) \int_0^\infty l \chi (l) \sin l dl \vspace{0.1cm}\\

		0 &
		\partial_w f (\nu_0, \overline{w}_{\nu_0}) \int_0^\infty l \chi(l) \cos l dl
		\end{pmatrix}
		\begin{pmatrix} 
		\tilde{\nu}\\\tilde{\omega}
		\end{pmatrix}
		=\begin{pmatrix}
	    c_1 \\ c_2
		\end{pmatrix},
		\end{equation*} 

		To solve this linear system, let us show that our assumptions for the characteristic equation , namely Assumption \ref{ASS3.7}, ensures that the determinant of the above matrix is non-zero, that reads as
		\begin{equation}\label{4.11}
		-\partial_w f (\nu_0, \overline{w}_{\nu_0}) \left[ \partial_\nu \partial_w f (\nu_0, \overline{w}_{\nu_0}) + \partial_w ^2 f (\nu_0, \overline{w}_{\nu_0}) \partial_\nu \overline{w}_{\nu_0} \right] \int_0^\infty \chi(l) \cos l dl \int_0^\infty l \chi(l) \cos l dl \neq 0.
		\end{equation}
		To check this property, recalling \eqref{4.8}, it is sufficient to check that
%		\eqref{4.8} ensures, 
%	$$
%	\partial_w f (\nu_0, \overline{w}_{\nu_0})  \int_0^\infty \chi(l) \cos l dl=1,
%$$	
%		so that to to check \eqref{4.11}, 
			\begin{equation*}
		\left[\partial_\nu \partial_w f (\nu_0, \overline{w}_{\nu_0}) + \partial_w^2 f (\nu_0, \overline{w}_{\nu_0})     \partial_\nu \overline{w}_{\nu_0} \right] \int_0^\infty l \chi(l) \cos(l) dl \neq 0.
		\end{equation*}
		Next set for $\theta\in\R$,
		$$ 	 
		\widehat{\chi}(\theta) = \int_0^\infty \chi(l) e^{-i \theta l} dl, 
		$$ 
and observe that the above condition (or equivalently \eqref{4.11}) is equivalent to
\begin{equation}\label{4.12}
\left[\partial_\nu \partial_w f (\nu_0, \overline{w}_{\nu_0}) + \partial_w^2 f (\nu_0, \overline{w}_{\nu_0})     \partial_\nu \overline{w}_{\nu_0} \right] {\rm Im}\,\widehat{\chi}'(1) \neq 0
\end{equation}
		On the other hand, by differentiating the characteristic equation \eqref{3.2} with respect to $\nu$ at $\nu=\nu_0$ (and recalling that $\omega_0=1$), we have
		
		\begin{equation*}
		\frac{d\lambda(\nu_0)}{d\nu} \partial_w f (\nu_0, \overline{w}_{\nu_0}) \int_0^\infty l \chi(l) e^{-i\,l} dl = \left[ \partial_\nu \partial_w f (\nu_0, \overline{w}_{\nu_0}) + \partial_w ^2 f (\nu_0, \overline{w}_{\nu_0}) \partial_\nu \overline{w}_{\nu_0} \right] \widehat{\chi}(1),
		\end{equation*}

which implies

\begin{equation*}
	\frac{d\lambda(\nu_0)}{d\nu} = \frac{ \left[ \partial_\nu \partial_w f (\nu_0, \overline{w}_{\nu_0}) + \partial_w ^2 f(\nu_0, \overline{w}_{\nu_0}) \partial_\nu \overline{w}_{\nu_0} \right] \int_0^\infty \chi(l) e^{-i l} dl} { \partial_w f (\nu_0, \overline{w}_{\nu_0}) \int_0^\infty l \chi(l) e^{-il} dl}.
\end{equation*}
Thus the transversality condition, that is Assumption \ref{ASS3.7}-(iv), becomes

\begin{equation*}
	\begin{array}{l}
		{\rm Re}  \dfrac{d\lambda(\nu_0)}{d\nu} = {\rm Re} \left\{  \dfrac{\left[\partial_\nu \partial_w f (\nu_0, \overline{w}_{\nu_0} ) + \partial_w ^2 f (\nu_0, \overline{w}_{\nu_0}) \partial_\nu \overline{w}_{\nu_0} \right] \widehat{\chi}(1)}  {i  \partial_w  f  (\nu_0, \overline{w}_{\nu_0})  \widehat{\chi}' (1) }\right\} \vspace{0.2cm} \\

		={\rm Re}\left\{ \left[\partial_\nu \partial_w f (\nu_0, \overline{w}_{\nu_0}) + \partial_w ^2 f (\nu_0, \overline{w}_{\nu_0})  \partial_\nu  \overline{w}_{\nu_0} \right] \underbrace{ \widehat{\chi}  (1)  \partial_w  f  (\nu_0, \overline{w}_{\nu_0})}_{=1}  \overline{i  \widehat{\chi}' (1)  } \right\}\vspace{0.2cm} \\

		=\left[\partial_\nu \partial_w f  (\nu_0, \overline{w}_{\nu_0})  +  \partial_w ^2  f  (\nu_0, \overline{w}_{\nu_0})  \partial_\nu \overline{w}_{\nu_0} \right] {\rm Im}\,  \left(\widehat{\chi}' (1) \right) \neq 0,
	\end{array}		
\end{equation*}
therefore \eqref{4.12} holds true and we can find a unique $\left(\tilde{ \nu}, \tilde{\omega}\right)\in\R^2$ solving the above two-dimensional linear system.  This completes the proof of the lemma. 

\end{proof}

\noindent \textbf{Last part of the proof of Theorem \ref{TH4.1}: } To conclude the proof of the Hopf bifurcation Theorem \ref{TH4.1}, we apply the implicit function theorem (see Deimling \cite[Theorem 15.2]{deimling2010nonlinear}) to the function $h: \mathbb{R}^3\times X_2\rightarrow C_{2\pi}(\mathbb{R})$ and we deduce that there exists a $C^1-$mapping $(\omega, \nu, z):(-\delta, \delta)\rightarrow\mathbb{R}^2\times X_2$, for some $\delta>0$ small enough, such that
$$ 
h(s, \omega(s), \nu(s), z(s))=0, \quad \forall s\in(-\delta, \delta). 
$$
By the definition of $h$, this is equivalent to say that
$$ 
F(\omega(s), \nu(s), \overline{w}_{\nu(s)}+s(u_1+z(s)))=0, 
$$
when $s\neq0$ with $(\omega(0), \nu(0), z(0))=(1, \nu_0, \overline{w}_{\nu_0})$. We see that $(\omega(s), \nu(s), \overline{w}_{\nu(s)}+s(u_1+z(s)))$ is the desired curve of solutions of $F=0$. Thus the proof is complete.

\end{proof}

\end{document}